\documentclass[12pt]{article}  % standard LaTeX, 12 point type 
\usepackage{amsfonts,latexsym} % for standard LaTeX
\usepackage{amsthm}
\usepackage{amssymb, amsmath}

\usepackage{graphicx,color}
\usepackage{listings}
\usepackage{algorithm}
\usepackage{algorithmic}
\usepackage{cancel} % for not tioco
\usepackage{wrapfig}
\newtheorem{theorem}{Theorem}[section]
\newtheorem{proposition}[theorem]{Proposition}
\newtheorem{lemma}[theorem]{Lemma}

\theoremstyle{definition}
\newtheorem{defn}{Definition}[section] 
\newtheorem{example}{Example}[section]

\theoremstyle{remark}
\newtheorem*{remark}{Remark}

\title{On the intersection of subgroups in free groups:
echelon subgroups are inert}

\author{Amnon Rosenmann \\ AIT Austrian Institute of Technology}
\date{}         

\begin{document}

% examples of some useful macros:

\newcommand{\CC}{\mathbb C} % blackboard math , for ``complex,'' etc

%\newcommand{\qed}{$\Box$}      % box  indicating end of proof.

% for a sequence of unnumbered displayed equations:
\newcommand{\beas}{\begin{eqnarray*}} 
\newcommand{\eeas}{\end{eqnarray*}} 

\newcommand{\bm}[1]{{\mbox{\boldmath $#1$}}} % for boldface math symbols 

% for binomial coefficients (AmS-LaTeX)
\newcommand{\bc}[2]{\genfrac{(}{)}{0pt}{}{#1}{#2}}

% for two rows, say under a summation sign (AmS-LaTeX)
\newcommand{\tworow}[2]{\genfrac{}{}{0pt}{}{#1}{#2}}

\maketitle

\begin{abstract}
A subgroup $H$ of a free group $F$ is called inert in $F$ if 
$\mathrm{rk}(H \cap G) \leq \mathrm{rk} (G)$ for every $G < F$.
In this paper we expand the known families of inert subgroups.
We show that the inertia property holds for 1-generator endomorphisms.
Equivalently, echelon subgroups in free groups are inert.
An echelon subgroup is defined through a set of generators that are in echelon form with respect to some ordered basis of the free group, and may be seen as a generalization of a free factor.
For example, the fixed subgroups of automorphisms of finitely generated free groups are echelon subgroups.
The proofs follow mostly a graph-theoretic or combinatorial approach.
\end{abstract}

%\keywords{Free groups, subgroups intersection, echelon subgroups, inert subgroups, compressed subgroups,
%1-generator endomorphisms, fixed subgroups of automorphisms}

%\classification{20E05, 20E07, 20E36}

%%%%%%%%%%%%%%%%%%%%%%%%%%%%%%%%%%%%%%%%%%%%%%%%%%%%%%%%%%%%%%%%

\section{Introduction}
\label{sec1}
The rank of the intersection of finitely generated subgroups of free groups is one of the topics of interest
in combinatorial group theory since Howson~\cite{H54} showed that when two subgroups are finitely generated then
so is their intersection. 
The famous Hanna Neumann Conjecture (Neumann~\cite{HN56}, \cite{HN57}) states that if $G_1$ is of rank $r_1>0$ and $G_2$ is of rank $r_2>0$ then the rank of $G_1 \cap G_2$ is at most $1 + (r_1-1)(r_2 - 1)$.
Tardos~\cite{T92} proved the conjecture for $r_1 \leq 2$.
Dicks and Formanek~\cite{DF01} improved it to $r_1 \leq 3$.
The conjecture was recently proved by Mineyev~\cite{M12} (see also Dicks~\cite{D11} for a simplified proof) and independently by Friedman~\cite{FD14} (including a simplified proof by Dicks). 

In this paper we study special types of subgroups of free groups and show that their involvement in
the intersection leads to a bound which is sharper than the general bound stated in the Hanna Neumann conjecture.
Dicks and Ventura~\cite{DV96} introduced the notion of inertia: a subgroup of a free group is called \emph{inert} if its intersection with any subgroup $G$ is of rank which is bounded by the rank of $G$.
Note that the inertia property is transitive: given a free group $F$ with subgroups $H < G < F$ then if $G$ is inert in $F$ and $H$ is inert in $G$ then $H$ is inert in $F$.

By Tardos~\cite{T92} every subgroup of rank 2 in a free group is inert: this is exactly the Hanna Neumann Conjecture when referring to subgroups of rank 2.
More natural examples of inert subgroups of free groups are free factors.
Dicks and Ventura~\cite{DV96} proved that the subgroup of a free group which is fixed by a family of injective endomorphisms of the free group is inert.
As for a family of general endomorphisms, not necessarily injective ones, it is still an open problem whether the fixed subgroup is inert.
Bergman~\cite{B99} showed that the rank of the fixed subgroup in this case is at most the rank of the free group, and his result was improved by Martino and Ventura~\cite{MV04a} to show that the fixed subgroup is compressed (a compressed subgroup is one which cannot have rank greater than the rank of a subgroup containing it; in particular, inert subgroups are compressed).

In Section~\ref{sec3} we introduce the notion of an echelon subgroup of a free group, which is defined through a set of generators that are in echelon form with respect to some ordered basis of the free group.
The class of echelon subgroups includes the class of free factors.
We show that every echelon subgroup $H$ of a free group $F$ is inert (Theorem~\ref{th2}) 

Echelon subgroups can be constructed in an iterative process through simple 1-generator subgroup endomorphisms.
Such endomorphisms fix all but (possibly) one of the elements of a free basis of the subgroup.
In Section~\ref{sec2} we define these endomorphisms and show in Theorem~\ref{th1} (which is, in fact equivalent to Theorem~\ref{th2}) that the image of a 1-generator endomorphism of a free group $F$ is inert in $F$.
The iterative process of 1-generator subgroup endomorphisms may also be used to construct non-echelon subgroups, which are still inert, as demonstrated by the subgroup of rank 3 presented in Example~\ref{ex1}.

Section~\ref{sec4} deals with the class of fixed subgroups of automorphisms of finitely-generated free groups.
These subgroups are inert as proved by Dicks and Ventura~\cite{DV96}.
Based on a structure theorem given by Martino and Ventura~\cite{MV04b} it is clear that the fixed subgroups form a special type of echelon subgroups, and since echelon subgroups are inert we have here another proof of the inertia property of the fixed subgroups of automorphisms of free groups.

We conclude in Section~\ref{sec5} with some open problems.

The proofs take mostly the graph-theoretic approach as done by Imrich~\cite{I77} and others to follow. 
\section{1-generator endomorphisms}
\label{sec2}
We start with some notation and definitions.
Let $F_n$ be the free group of rank $n$.
When $\{x_1, \ldots , x_n \}$ is a set of free generators for $F_n$
then we write it as $F_n = \langle x_1, \ldots , x_n \rangle$.
The same notation applies to subgroups of $F_n$.

When $H = \langle y_1, \ldots , y_m \rangle$ and $K < H < F_n$ then we denote by $\Gamma_H(K)$ the \emph{Schreier coset graph} of $K$ in $H$ with respect to the basis $\{y_1, \ldots, y_m \}$.
The root of this graph represents the coset $K1$ (1 being the trivial element), and for each vertex $v \in V(\Gamma_H(K)) = K \backslash H$ there are $2m$ directed edges labelled $y_1^{\pm 1}, \ldots, y_m^{\pm  1}$ going out of $v$.
For each edge $e$ labelled $y_j$ the vertex $v = \iota(e)$ is the initial vertex of $e$, and the vertex
$w = \tau(e) = vy_j$ is the terminal vertex of $e$.
In the other direction this edge is denoted $\bar{e}$ and labelled $y_j^{-1}$ with $w = \iota(e)$ and $v = \tau(e)$.
If $\gamma$ is a path starting at the root of $\Gamma_H(K)$ then the word 
$w = y_{i_1}^{\pm 1} \cdots y_{i_s}^{\pm  1}$ that is read off along the path represents
an element of $K$ if and only if $\gamma$ is a closed path (cycle) that terminates at the root.
In general, if two right cosets $Kh$ and $Kh'$ are equal then the two paths in $\Gamma_H(K)$ that start at the root with edge labels that form the two words $h$ and $h'$ end at the same vertex of $\Gamma_H(K)$.

The \emph{core} of the graph $\Gamma_{H}(K)$ is the minimal connected subgraph containing all non-trivial reduced (without cancellation) cycles (the infinite hanging trees are chopped) (see Stallings~\cite{S83} for more details).

As is known, $\mathrm{rk}(K) = b_1(\Gamma_H(K))$, where $b_1$ represents the \emph{first Betti number}, or the cyclomatic number (number of cycles), of $\Gamma_H(K)$.
By the very definition of the core, it is clear that $b_1$ may be confined to the core of the graph.
Note that $\mathrm{rk}(K) < \infty$ if and only if the core of $\Gamma_{H}(K)$ is finite.

Throughout the paper we will be dealing with basic simple endomorphisms that we call 1-generator subgroup endomorphisms.
\begin{defn}[1-generator subgroup endomorphism]
An endomorphism $\phi : H \to H$ of a subgroup $H < F_n$ that fixes all but (possibly) one of the members of a set of free generators of $H$ is a \emph{1-generator subgroup endomorphism}.
\end{defn}
\begin{defn}[Inert endomorphism]
An endomorphism $\phi : H \to H$ is \emph{inert} if its image $H \phi$ is inert in $H$.
\end{defn}
\begin{theorem}
A 1-generator endomorphism of a free group is inert.
\label{th1}
\end{theorem}
\begin{proof}
It suffices to confine ourselves to finitely generated free groups.
So let $F_n = \langle x_1, \ldots , x_n \rangle$ be a free group of rank $n$.
Let $\phi : F_n \to F_n$ be a 1-generator endomorphism defined by $x_i \phi = x_i$, $i = 1, \ldots, n-1$, $x_n \phi = x \in F_n$ and let $H = F_n \phi$.
Let $G < F_n$ be some finitely generated subgroup of $F_n$ with $\Gamma_{F_n}(G)$ the Schreier coset graph of $G$ in $F_n$ with respect to the basis $\{x_1, \ldots, x_n \}$.
Let $K = H \cap G < H$ and assume that $K$ is non-trivial (otherwise the proof is trivial).

If $x \in F_{n-1}$ then $\phi$ is a retraction with $H = F_{n-1}$ and $K$ a free factor of $G$.
Hence $\mathrm{rk}(K) \leq \mathrm{rk} (G)$.

So suppose that $x \notin F_{n-1}$.
Thus $\{x_1, \ldots, x_{n-1}, x\}$ is a free basis of $H$.
By applying Nielsen transformations, if necessary, we may assume that the reduced word representing $x$ starts and ends with $x_n^{\pm 1}$.
Furthermore, replacing $x_n$ by $x_n^{-1}$ in the initial basis, if necessary, we may assume that $x$ starts with $x_n$.

Let $\Gamma_{H}(K)$ be the Schreier coset graph of the right cosets of $K$ in $H$ with respect to the generators
$ x_1, \ldots , x_{n-1}, x$.
There is a natural injective $H$-map $i$ from the quotient set $K \backslash H$
to the quotient set $G \backslash F$, $Kh \mapsto Gh$, which induces the following continuous mapping $\alpha : \Gamma_{H}(K) \to \Gamma_{F_n}(G)$.
On $V(\Gamma_{H}(K))$, the set of vertices of $\Gamma_{H}(K)$, $\alpha$ is the natural embedding $i$. 
Similarly for the edges in $E(\Gamma_{H}(K))$ with labels
$y \in \{ x_1^{\pm 1}, \ldots, x_{n-1}^{\pm 1} \}$.
Then an edge $e$ with label $x^{\pm 1}$ and with initial vertex $v = \iota(e)$ and terminal
vertex $w = \tau(e)$ is mapped to the concatenation of edges along the path that leads from $\alpha(v)$ to $\alpha(w)$, where the labels of the edges that are read off along this path form the reduced word $x^{\pm 1}$.

We will handle separately the case where $x$ starts and ends with $x_n$ and the one where $x$ starts with $x_n$ and ends with $x_n^{-1}$.
Let us begin with $x$ starting and ending with $x_n$.
Denote by $C(G)$ the core of $\Gamma_{F_n}(G)$ and by $C(K)$ the core of $\Gamma_{H}(K)$.
The mapping $\alpha : \Gamma_{H}(K) \to \Gamma_{F_n}(G)$ restricts to a mapping
$\alpha : C(K) \to C(G)$.
Indeed, any reduced path in $\Gamma_{H}(K)$ is mapped by $\alpha$ to a reduced path in
$\Gamma_{F_n}(G)$ because no cancellation occurs in
$x^2$.
In particular, reduced cycles of $C(K)$ are mapped to reduced cycles of $C(G)$.

The mapping $\alpha : C(K) \to C(G)$ induces a (discontinuous in general) injective mapping 
$\tilde{\alpha} : C(K) \to C(G)$.
The difference between $\alpha$ and $\tilde{\alpha}$ is that an $x$-edge $e$ is mapped to the $x_n$-edge which is the first edge in the path $\alpha(e)$, and similarly, an $x^{-1}$-edge $e'$ is mapped to the $x_n^{-1}$-edge which is the first edge in the path $\alpha(e')$.
Thus, $\tilde{\alpha}$ is an injective mapping on both the vertices and the (directed) edges of $C(K)$.
Each vertex $v \in V(C(K))$ is mapped to a unique vertex $\tilde{\alpha}(v) \in C(G)$, and each edge
$e \in E(C(K))$ with $\iota(e) = v$ is mapped to a unique edge $\tilde{\alpha}(e)$ with
$\iota(\tilde{\alpha}(e)) = \tilde{\alpha}(v)$.
This implies that the degree of each vertex $v \in C(K)$ is at most the degree of the corresponding vertex
$\tilde{\alpha}(v) \in C(G)$. 
By the formula for the first Betti number of the core $C(G)$ (which equals the rank of $G$)
\begin{equation}
\mathrm{rk} (G) = b_1(C(G)) = 1 + \frac{\sum_{v \in C(G)}(\mathrm{deg}(v) - 2)}{2}
\label{frml1}
\end{equation}
and similarly for $\mathrm{rk} (K) = b_1(C(K))$, we get that $\mathrm{rk}(K) \leq \mathrm{rk} (G)$.

It remains to handle the case where $x$ starts with $x_n$ and ends with $x_n^{-1}$.
That is, $x$ is of reduced form $x_n p q p^{-1} x_n^{-1}$, where $p$ may be trivial and $q$ is non-trivial and cyclically reduced.
The coset graphs, cores and mapping $\alpha : \Gamma_{H}(K) \to \Gamma_{F_n}(G)$ are as before.
But now the set of vertices of $C(K)$ does not necessarily embed in the set of vertices of $C(G)$.
This is because an $x$-path in $\Gamma_{F_n}(G)$ followed by another $x$-path has to backtrack, hence it may go
beyond the boundary of $C(G)$ in the first $x$-path before returning to $C(G)$ in the second $x$-path.
So we extend $C(G)$ to a graph $\bar{C}(G)$ which contains these extra vertices and edges that we call "hairs" (each such hair has edges with labels that create a word which is a suffix of $p^{-1} x_n^{-1}$ when read outward).

The proof in the previous case was based on counting the degrees of vertices.
But now both $x$ and $x^{-1}$ start with an $x_n$-edge, and one needs to get into a more detailed and meticulous examination of the structure of $x$ when using degree counting as a method for the proof.
Instead, we are counting cycles, which turns out to be a simpler mission.
The idea is to start with the extended core $\bar{C}(G)$ and at each step remove an $x_n$-edge and add an $x$-edge while preserving the initial cyclomatic number of $C(G)$ as a bound for the cyclomatic number of the evolving graph.
After finitely-many steps we reach a graph in which $C(K)$ is embedded, thus having a cyclomatic number which is greater or equal to the cyclomatic number of $C(K)$ and at the same time being less or equal to the cyclomatic number of $C(G)$. 

So, suppose first that $v$ is a vertex of $C(K)$ whose image $\alpha(v)$ is in $C(G)$.
Suppose also that $C(K)$ contains an $x$-edge going out of $v$ and ending at $w$ and that $\alpha(w) \in C(G)$.
Then we remove the $x_n$-edge going out of $\alpha(v) \in C(G)$ and add an $x$-edge starting at $\alpha(v)$ and ending at $\alpha(w)$.
If the $x_n$-edge lies on a simple cycle then by removing it the graph remains connected and the cyclomatic number decreases by 1, and by adding the $x$-edge the cyclomatic number increases by 1.
If the $x_n$-edge does not belong to any simple cycle but still lies on a reduced cycle it means that $C(G)$ is of the form of two components $A$ and $B$, both containing cycles, and a single "bridge" connecting them, with the $x_n$-edge being part of the bridge.
By removing the $x_n$-edge the graph becomes disconnected.
The addition of the $x$-edge may then form again one connected component, so keeping the cyclomatic number unchanged.
Otherwise, suppose that $\alpha(v)$ is on the bridge (possibly a boundary vertex of $A$) and $\alpha(w)$ is in the union of $A$ and the part of the bridge between $A$ and $\alpha(v)$.
Suppose also that the $x$-path that starts at $\alpha(v)$ passes the bridge towards $B$, makes at least one loop in $B$, returns on the bridge towards $A$ and ends at $\alpha(w)$ (with possible more visits to $A$ and $B$ in  between).
In this case the removal of the $x_n$-edge decreases the cyclomatic number of the evolving graph by $b_1(B) \geq 1$, the number of cycles in $B$, and the addition of the $x$-edge increases it by 1.
Clearly, the cyclomatic number of the new graph is at most that of the one in the previous step.
At the worst case it may happen that the component $B$ will join $A$ again by another removal of an $x_n$-edge and addition of an $x$-edge.
This will add $b_1(B) - 1$ to the cyclomatic number, so gaining what was lost before.

Note that the removing of the $x_n$-edge may result in a new hair added to the graph.
So, it may happen that the next time we remove an $x_n$ edge it resides in such a hair.
We claim that in this case the direction of the $x$-path starting at $\alpha(v)$ is towards the base of the hair.
Otherwise, we get two $x$-paths proceeding one towards the other and overlapping in a way that one $x$-path starts with a reduced word $r$ while the other starts with a reduced word $r^{-1}$ - which is impossible.
But when the direction of the $x$-path is towards the base of the hair then both operations of removing an edge as well as adding an edge leave the cyclomatic number unchanged.

Finally, suppose that $v$ is a vertex of $C(K)$ whose image $\alpha(v)$ is in a hair in $\bar{C}(G) - C(G)$ and there is an $x$-edge going out of $v$ in $C(K)$.
Assume there exists an integer $j > 0$ such that $v' = v x^{-j}$ and $\alpha(v') \in C(G)$, that is, $\alpha(v')$ is not inside a hair.
Then we make $j$ steps of removing an $x_n$-edge and adding an $x$-edge, starting at $\alpha(v')$, then at $\alpha(v'x)$ until $\alpha(v'x^j=v)$.
At each such step we remove an $x_n$-edge which is part of a simple cycle (decreasing the cyclomatic number by 1) and add an $x$-edge (thereby increasing again the cyclomatic number), and by doing so, when we reach $\alpha(v)$ it is no longer part of a hair.

When there is no such $j > 0$ for which $v' = v x^{-j}$ with $\alpha(v') \in C(G)$ then necessarily $C(K)$ consists of a single cycle with all its edges labelled $x$.
Then $C(G)$ contains the corresponding cycle with edges labelled $q$ (remember that $x = x_n p q p^{-1} x_n^{-1}$) and the claim of the theorem holds in this case too.
\end{proof}
\section{Echelon subgroups}
\label{sec3}
An echelon subgroup of a free group is the image of a special endomorphism of the free group which in some sense reminds an operator in a vector space represented by a matrix in echelon form .
\begin{defn}[Echelon form]
Let $F_n$ be the free group of rank $n$ with an ordered basis $X = \{ x_1, \ldots , x_n \}$ and let $F_i = \; \langle x_1, \ldots , x_i \rangle$, $i = 0, \ldots, n$, with $F_0 = \; \langle 1 \rangle$.
We say that a subgroup $H < F_n$ is in \emph{echelon form} with respect to $X$ if
$\mathrm{rk}(H \cap F_i) - \mathrm{rk} (H \cap F_{i-1}) \le 1$ for each $i$, $i=1, \ldots ,n$.
\end{defn}
Note that always $\mathrm{rk}(H \cap F_i)$ is greater or equal to $\mathrm{rk} (H \cap F_{i-1})$ since the latter is a free factor of the former.
\begin{remark}
$H < F_n$ is in echelon form with respect to $\{ x_1, \ldots , x_n \}$ if and only if $H$ has a free ordered basis
$\{y_{i_1}, \ldots ,y_{i_r} \}$ with $y_{i_j} \in F_{i_j} - F_{i_j-1}$,
$1 \leq i_1 < i_2 < \cdots < i_r \leq n$ and $F_{i_j} = \; \langle x_1, \ldots , x_{i_j} \rangle$, $j = 1, \ldots, r$.
That is, each $y_{i_j}$ contains at least one new (not present in the previous basis elements) generator of $F_n$ in its reduced form.
\end{remark}
\begin{defn}[Echelon subgroup]
A subgroup $H < F_n$ is an \emph{echelon subgroup} of $F_n$ if $H$ is in echelon form with respect to \emph{some} free ordered basis of $F_n$.
\end{defn}
\begin{remark}
If $F_n$ is a free group of rank $1 < n < \infty$ then an echelon subgroup $H < F_n$ that is of finite index in $F_n$ must be $F_n$ itself by Schreier Index Formula (see e.g. Lyndon and Schupp~\cite{LS77}).
\end{remark}
\begin{lemma}
Every echelon subgroup of $F_n$ is the image of $F_n$ by an endomorphism $\phi$ which is the result of performing
at most $n$ 1-generator subgroup endomorphisms, and for every 1-generator endomorphism of $F_n$ its image is an echelon subgroup of $F_n$.
\label{lem1}
\end{lemma}
\begin{proof}
Let $H < F_n$ be an echelon subgroup of $F_n$ with respect to the ordered basis $\{ x_1, \ldots , x_n \}$ of $F_n$.
Then there exists an ordered basis $\{y_{i_1}, \ldots ,y_{i_r} \}$ of $H$, with $y_{i_j} \in F_{i_j} - F_{i_j-1}$, 
$1 \leq i_1 < i_2 < \cdots < i_r \leq n$, and $F_{i_j} = \; \langle x_1, \ldots , x_{i_j} \rangle$, $j = 1, \ldots, r$.

We form a series of subgroups $H_k$, $k = 0, \ldots, n$ through 1-generator subgroup endomorphisms $\phi_k$.
We start with $H_0 = F_n$ and obtain the subgroup $H_1 = H_0 \phi_1$ in the following way.
If $i_r = n$ then we map $x_n$ to $y_{i_r}$, while fixing the other generators.
The image is the subgroup $H_1 = \; \langle x_1, \ldots , x_{n-1}, y_{i_r} \rangle$.
If $i_r \neq n$ then $x_n$ is mapped to the trivial group element and then $H_1 = \; \langle x_1, \ldots, x_{n-1} \rangle$. 

Continuing in the same manner, at step $k$ we form the subgroup $H_k = H_{k-1} \phi_k$, where $\phi_k$ is
the 1-generator endomorphism of $H_{k-1}$ which fixes all generators of $H_{k-1}$ different from $x_{n+1-k}$,
and
$$
x_{n+1-k} \phi_k = 
\left\{
\begin{array}{ll}
y_{n+1-k} & \mbox{if $n+1-k \in \{i_1, \ldots , i_r \}$} \\
1 & \mbox{otherwise} \\
\end{array}
\right.
$$
$H_k < H_{k-1}$ is freely generated by the elements $x_1, \ldots, x_{n-k}$ and the elements $y_{i_j}$
for which $i_j \geq n-k+1$. 
Finally, at step $n$ the free ordered basis $\{y_{i_1}, \ldots ,y_{i_r} \}$ of the subgroup $H_n = H$ is constructed.
If we skip the endomorphisms $\phi_k$ which are the identity then the number of steps may be less than $n$.

For the second part of the claim, let $\phi : F_n \to F_n$ be a 1-generator endomorphism and let $H = F_n \phi$.
By possibly renaming the ordered set of generators of $F_n$, we may assume that $\phi$ is defined by $x_i \phi = x_i$, $i = 1, \ldots, {n-1}$, $x_n \phi = x \in F_n$.
Then either $x \in F_{n-1}$ and $H = F_{n-1}$, or else $H = \langle x_1, \ldots , x_{n-1}, x \rangle$.
In both cases $H$ is an echelon subgroup of $F_n$.
\end{proof}
\begin{theorem}
Echelon subgroups of free groups are inert.
\label{th2}
\end{theorem}
\begin{proof}
By Lemma~\ref{lem1} an echelon subgroup $H < F_n$ can be reached through a series of 1-generator subgroup endomorphisms.
Surely, $F_n$ is inert in itself, and by Theorem~\ref{th1} and the transitivity of the inertia property $H$ is inert in $F_n$.
\end{proof}
On the other hand, by the second part of Lemma~\ref{lem1}, the inertia property of echelon subgroups implies that 1-generator endomorphisms are inert.
So, in fact, Theorem~\ref{th2} is equivalent to Theorem~\ref{th1}.
\begin{example}
In this example we construct a non-echelon subgroup $H$ of $F_3$ through a series of 1-generator subgroup endomorphisms.
That is, in general, if $G$ is an echelon subgroup of $F_n$ and $H$ is an echelon subgroup of $G$ then $H$ is not
necessarily echelon in $F_n$.
So let $F_3$ be the free group with ordered basis $\{ x, y, z \}$.
We apply the 1-generator endomorphism defined by $x \mapsto u = x^2 y^2 x^2$ to form the echelon subgroup $G = \langle u=x^2 y^2 x^2, y, z \rangle$.
Then we perform the 1-generator endomorphism of $G$ given by $y \mapsto v = y^2 z^2 y^2$ to obtain the subgroup $K = \langle u=x^2 y^2 x^2, v = y^2 z^2 y^2, z \rangle$.
Finally, with $z \mapsto w = z^2 u z^2$ we obtain the subgroup $H = \langle u=x^2 y^2 x^2, v = y^2 z^2 y^2, w = z^2 x^2 y^2 x^2 z^2 \rangle$.
$H$ is not echelon: since it is of the same rank as $F_3$, in order to be echelon it must contain a positive power of a primitive element.
But if $g^i \in H$ then $g \in \langle x^2, y^2, z^2 \rangle$ and such elements are known to be non-primitive. 
Since $H$ was constructed by 1-generator subgroup endomorphisms then it is, however, inert in $F_n$. 
\label{ex1}
\end{example}
The notion of compression seems to be weaker than the notion of inertia.
When $H$ is inert then the rank of the intersection of $H$ with every subgroup $G < F$ is at most the rank of $G$, while when $H$ is compressed the same property is confined to subgroups $G$ containing $H$.
It is, however, conjectured that the two notions coincide (see~\cite{V02} for a discussion on this conjecture).
Typical examples of compressed subgroups are retracts.
Since echelon subgroups are inert then we know they are compressed.
A direct proof of the compression property for echelon subgroups is, however, simple so we bring it below.  
\begin{proposition}
Echelon subgroups of free groups are compressed.
\label{pr1}
\end{proposition}
\begin{proof}
Let $H$ be an echelon subgroup of $F_n$ with respect to the ordered basis $X = \{ x_1, \ldots , x_n \}$ of $F_n$,
and let $F_i = \; \langle x_1, \ldots , x_i \rangle$, $i = 0, \ldots, n$, with $F_0 = \; \langle 1 \rangle$.
Let $H < G < F_n$. We need to show that $\mathrm{rk}(G) \geq \mathrm{rk}(H)$.
For each $i$, let $H_i = H \cap F_i$ and let $G_i = G \cap F_i$.
Then $H_i$ (respectively $G_i$) is a free factor of $H$ (respectively $G$) and also a free factor
of $H_{i+1}$ (respectively $G_{i+1}$).

We claim that $\mathrm{rk}(G_i) \geq \mathrm{rk}(H_i)$ for each $i$.
It certainly holds for $i = 0$.
Suppose by induction the claim is true for $i = k$.
By the very definition of echelon subgroups, $s_{k+1} = \mathrm{rk}(H_{k+1}) - \mathrm{rk} (H_k) \le 1$.
If $s_{k+1} = 0$ then of course $\mathrm{rk}(G_{k+1}) \geq \mathrm{rk}(H_{k+1})$.
Otherwise, $\mathrm{rk}(H_{k+1}) = \mathrm{rk} (H_k) +1$, and let $h \in H_{k+1} - H_k$.
Since $G_k$ is a free factor of $G_{k+1}$ then $\mathrm{rk}(G_{k+1}) \geq \mathrm{rk} (G_k)$.
Moreover, $h \in G_{k+1}$ since $H_{k+1} < G_{k+1}$, and $h \notin G_k$ since $h \in F_{k+1} - F_k$.
We conclude that $\mathrm{rk}(G_{k+1}) > \mathrm{rk} (G_k)$, and by the induction hypothesis
$\mathrm{rk}(G_{k+1}) \geq \mathrm{rk}(H_{k+1})$.
In particular, for $k = n$ we get that $\mathrm{rk}(G) = \mathrm{rk}(G_n) \geq \mathrm{rk}(H_n) = \mathrm{rk}(H)$.
\end{proof}
\section{Example: subgroups fixed by automorphisms}
\label{sec4}
Given an automorphism $\varphi : F_n \to F_n$, a much studied object is $Fix(\varphi) < F_n$, the subgroup
consisting of the group elements fixed by $\varphi$.
Scott conjectured that $Fix(\varphi)$ is finitely-generated and, moreover, is of rank bounded by $n$.
His conjecture was proven to be true:
\begin{theorem} [Gersten \cite{G83}]
If $\varphi : F_n \to F_n$ is an automorphism then the rank of $Fix(\varphi)$ is finite.
\label{th3}
\end{theorem}
\begin{theorem} [Bestvina-Handel \cite{BH92}]
If $\varphi : F_n \to F_n$ is an automorphism then the rank of $Fix(\varphi)$ is at most $n$.
\label{th4}
\end{theorem}
Martino and Ventura gave a description of the structure of $Fix(\varphi)$, which can be stated as follows. 
\begin{theorem} [Martino-Ventura \cite{MV04b}]
For every automorphism $\varphi : F_n \to F_n$ there exists an ordered basis $\{ x_1, \ldots , x_n \}$ of $F_n$ such that
$$
Fix(\varphi) =  \; \langle y_1, \ldots , y_r, z_1, \ldots, z_s  \rangle,
$$
where the $y_j$ are not proper powers and
$$y_j \in \langle x_{i_{j-1}+1}, \ldots, x_{i_j} \rangle, \ i_{j-1} < i_j, \ 
j = 1, \ldots, r, \ r \geq 0, \ i_0 = 0,
$$
$$z_k = x_{i_r+k}^{-1} w_k x_{i_r+k}, \ w_k \in F_{i_r+k-1}, \ k = 1, \ldots, s, \ s \geq 0, \ 
i_r + s \leq n.
$$
\label{th5}
\end{theorem}
It is clear from the structure of $Fix(\varphi)$ that these subgroups form a special kind of echelon subgroups of $F_n$, hence by Theorem~\ref{th2} they are inert.
This is a known fact, already proved by Dicks and Ventura in 1996 (by other methods, of course).
\begin{theorem} [Dicks-Ventura \cite{DV96}]
The fixed subgroup $Fix(\varphi)$ of any automorphism $\varphi$ of $F_n$ is inert. 
\label{th7}
\end{theorem}
\section{Open problems}
\label{sec5}
We present here problems for further research.
\begin{itemize}
\item  Let $F$ be a free group defined by a set of free generators and let $H$ be a subgroup of $F$ defined by a set of finitely-many generators expressed in terms of the generators of $F$.
Is it algorithmically decidable whether $H$ is an echelon subgroup of $F$?
\item  Let $H$ be an echelon subgroup of a free group $F$ and let $G$ be a subgroup of $F$.
Is $K = H \cap G$ an echelon subgroup of $G$?
\item  Let $H, G$ be echelon subgroups of a free group $F$.
Is $K = H \cap G$ an echelon subgroup of $F$?
One may want to look first at the case where $H$ and $G$ are in echelon form with respect to the same ordered basis of $F$.
\end{itemize}

\end{document}